\documentclass[11pt, reqno]{amsart}
\usepackage[dvipsnames,usenames]{color}
\usepackage{amsmath}
\usepackage{amsfonts}
\usepackage{amssymb}
\usepackage{mathtools}
\usepackage{mathrsfs}
\usepackage{amsthm}
\usepackage{hyperref}
\usepackage{url}
\usepackage{comment}
\date{}
\newtheorem{theorem}{Theorem}
\newtheorem{thm}{Theorem}
\newtheorem{lem}{Lemma}

\theoremstyle{definition}
\newtheorem{dfn}{Definition}
\numberwithin{equation}{section}

\theoremstyle{remark}

\numberwithin{equation}{section}

\makeatletter
\@namedef{subjclassname@2020}{\textup{2020} Mathematics Subject Classification}
\makeatother

\begin{document}
	\setcounter{page}{1}

	\title[Combinatorial Calabi flow for ideal circle pattern]
	{Combinatorial Calabi flow for ideal\\ circle pattern}

\author[Shengyu Li and Zhigang Wang]{Shengyu Li and Zhigang Wang*}

\address{\noindent Shengyu Li \vskip.03in
School of Mathematics and Statistics, Hunan First Normal University, Changsha 410205, Hunan, P. R. China.}
	
	\email{lishengyu$@$hnu.edu.cn}

	\address{\noindent Zhigang Wang\vskip.03in
		School of Mathematics and Statistics, Hunan First Normal University, Changsha 410205, Hunan, P. R. China.}
	
	\email{wangmath$@$163.com}


	
\thanks{$^*$Corresponding author.}
	
	\date{\today}

	\begin{abstract}
We study the combinatorial Calabi flow for ideal circle patterns in both hyperbolic and Euclidean background geometry. We prove that the flow exists for all time and converges exponentially fast to an ideal circle pattern metric on surfaces with prescribed attainable curvatures. As a consequence, we provide an algorithm to find the desired ideal circle patterns.
	\end{abstract}
\maketitle

\tableofcontents	

\section{Introduction}
\subsection{Ideal circle pattern}
Circle patterns was introduced by Thurston \cite{T1} to study the geometry and topology of 3-manifolds.
Thurston \cite{T2} once proposed a conjecture that infinite hexagonal circle packings converge to classical Riemann mapping, which was proved by Rodin-Sullivan \cite{RS}.
Since then, circle patterns have been studied intensively in \cite{Liu,Schr1,S,Z1}.

Assume that $S$ is a compact oriented closed surface with a triangulation $\mathcal T$ and a constant Riemannian metric $\mu$.
A circle pattern $\mathcal P$ on $(S,\mu)$ is a collection of oriented circles.
A circle pattern $\mathcal P$ is $\mathcal T$-type if there exists a geodesic triangulation $\mathcal T(\mu)$ of $(S,\mu)$ with the following properties: 
(i) $\mathcal T(\mu)$ is isotopic to $\mathcal T$; (ii) the vertices of $\mathcal T(\mu)$ coincide with the centers of circles in $\mathcal P$. 

Let $V,E,F$ denote the sets of vertices, edges and faces of $\mathcal T$.
Connecting the centers of two adjacent circles $C_{v_{i}}$ and $C_{v_{j}}$ of circle pattern $\mathcal P=\{C_{v}:v\in V\}$ gives an edge $e=[v_{i},v_{j}]\in E$. For every $e\in E$, there exists an exterior intersection angle $\Theta(e)\in(0,\pi)$. For each $v\in V$, let $\mathbb{D}_v$ (resp. $D_v$) denote the open (resp. closed) disk bounded by $C_v$. Each connected component of the set $S\setminus(\cup_{v\in V} \mathbb{D}_v)$ is called an interstice.

\begin{dfn}
A $\mathcal T$-type circle pattern $\mathcal P$ on $(S,\mu)$ is called ideal if it satisfies the following properties: 

\begin{itemize}
  \item [(i)] there exists a one-to-one correspondence between the interstices of $\mathcal P$ and the 2-cells of $\mathcal T$; 
  \item [(ii)] every interstice of $\mathcal P$ consists of a point (see, e.g., Figure \ref{F11}).
\end{itemize}
\end{dfn}

\begin{figure}[htbp]
\centering
\includegraphics[height=4.0cm]{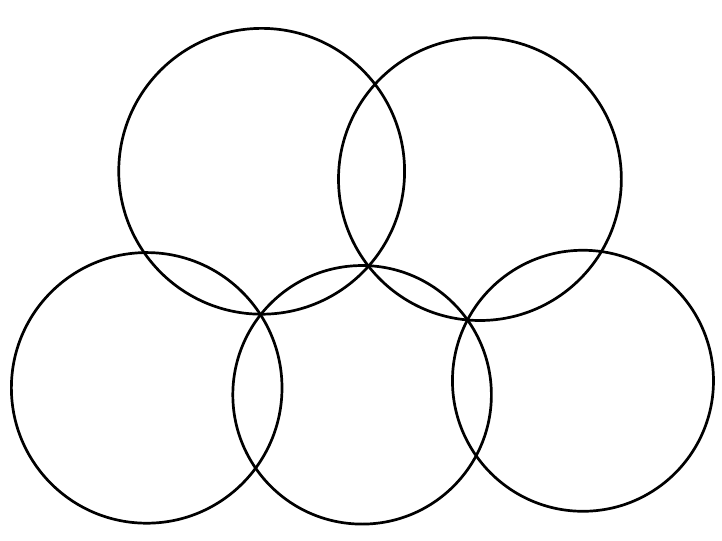}
\caption{An ideal circle pattern.}
\label{F11}
\end{figure}


Let $\Theta: E\to(0,\pi)$ be an angle function.
A natural question asks whether there exists a $\mathcal T$-type ideal circle pattern with a given exterior intersection function $\Theta$.
Thurston \cite[Chap.13]{T1} observed the relation between circle patterns and hyperbolic polyhedra.
Let $\mathbb{B}^{3}$ denote the hyperbolic 3-space.
Given a convex hyperbolic polyhedron $\mathcal Q$ in $\mathbb{B}^{3}$, the plane of each polyhedron face will cut a circle on the sphere $\partial\mathbb{B}^{3}$. All circles form a circle pattern with dual combinatorial type of $\mathcal Q$. The exterior intersection angles between adjacent circles are equal to the dihedral angles of corresponding faces of $\mathcal Q$.
A convex hyperbolic polyhedron in $\mathbb{B}^{3}$ is called ideal if all its vertices are on the sphere $\partial\mathbb{B}^{3}$.
In this way, an ideal circle pattern corresponds to a unique ideal polyhedron.
Consider the following condition:

\hspace{-4.55mm}($\mathbf{C1}$) For any distinct edges $e_{1},e_{2},\cdots,e_{m}$ forming the boundary of a 2-cell of $\mathcal T$, $$\sum_{i=1}^{m}\Theta(e_{i})=(m-2)\pi.$$

Let $g$ denote the genus of surface $S$.
Under the dual conditions of ($\mathbf{C1}$), Rivin \cite{R} proved that there exists a unique ideal polyhedron with the dihedral angles given by $\Theta$ for the case $g=0$. Equivalently, he showed that there exists a unique ideal circle pattern with given exterior intersection angles on surfaces.
Bobenko-Springborn \cite{BS} further generalized this result to the case $g\geq1$.

\subsection{Thurston's construction and curvature map}
Thurston's construction is presented as follows.
Let $S$ be a surface of genus $g>1$ (resp. $g=1$) with a triangulation $\mathcal T$.
Assume $\Theta(e)\in(0,\pi)$ for each $e\in E$.
Let $V=\big\{v_1,v_2,\cdots,v_{N}\big\}$.
A radius vector $r=(r_{1},r_{2},\cdots,r_{N})\in \mathbb{R}_{+}^{N}$ assigns each vertex $v_i\in V$ a positive radius $r_i$.
For each edge $[v_i,v_j]\in E$, we assign a length
\[
l_{ij}=\cosh^{-1}\big(\cosh r_{i}\cosh r_{j}+\sinh r_{i}\sinh r_{j}\cos\Theta([v_i,v_j])\big)
\]
in hyperbolic background geometry and
\[
l_{ij}=\sqrt{r_i^2+r_j^2+2r_ir_j\cos\Theta([v_i,v_j])}
\]
in Euclidean background geometry.
It is easy to check that there exists a unique geodesic triangle with lengths $l_{ij}$, $r_{i}$ and $r_{j}$.
Gluing all these hyperbolic (resp. Euclidean) triangles along the common edges gives a hyperbolic (resp. Euclidean) ideal circle pattern metric $\mu$ on the surface $S$ with possible cone singularities at vertices of $\mathcal T$.

To describe the singularities on $(S,\mu)$, the discrete curvature is defined to be
\[
K_i=2\pi-\sigma(v_i),
\]
where $\sigma(v_i)$ is the sum of inner angles at $v_i$ for all triangles incident to $v_i$.
This yields the following curvature map
\[
\begin{aligned}
Th: \quad {\mathbb{R}}_{+}^{N} \quad &\longrightarrow \quad\mathbb{R}^{N}\quad\\
\left(r_1,r_2,\cdots,r_{N}\right)&\longmapsto \left(K_1, K_2,\cdots,K_{N}\right).\\
\end{aligned}
\]
Then $K_1,K_2,\cdots, K_{N}$ are smooth functions of $r$.
Bobenko-Springborn~\cite{BS} gave complete characterizations of the images of the curvature map in both hyperbolic and Euclidean background geometry.

\begin{thm} [Bobenko-Springborn \cite{BS}]\label{T-1-1}
Assume that $\Theta$ is an angle function satisfying the condition $\mathbf{(C1)}$.
\begin{itemize}
  \item [(i)]
In hyperbolic background geometry, $Th$ is injective.
Moreover, the image of $Th$ consists of vectors $(K_{1},K_{2},\cdots,K_{N})$ satisfying
\begin{equation}
\label{E-1-1}
K_i<2\pi\quad (\forall\, i=1,2,\cdots,N),
\end{equation}
and
\begin{equation}
\label{E-1-2}
\sum\limits_{v_{i}\in A}K_{i}>2\pi|A|-2\sum\limits_{e\in E,\, \partial e\cap A\neq\emptyset}\Theta(e).
\end{equation}
for any non-empty subset $A$ of $V$.

\item [(ii)] In Euclidean background geometry, $Th$ is injective up to scaling.
Moreover, the image of $Th$ consists of vectors $(K_{1},K_{2},\cdots,K_{N})$ satisfying \eqref{E-1-1}
and
\begin{equation}\label{12}
\sum\limits_{v_{i}\in A}K_{i}\geq 2\pi|A|-2\sum\limits_{e\in E,\, \partial e\cap A\neq\emptyset}\Theta(e).
\end{equation}
for any non-empty subset $A$ of $V$, where the equality holds if and only if $A=V$.
\end{itemize}
\end{thm}
\subsection{Main result}
Chow-Luo \cite{ChowLuo-jdg} introduced the combinatorial Ricci flow on closed surfaces and proved the existence and convergence of the flow. This yields an algorithm finding circle patterns.
Ge~\cite{G1} and Ge-Hua~\cite{G2} introduced the combinatorial Calabi flow for circle patterns with non-obtuse exterior intersection angles in Euclidean and hyperbolic background geometry, respectively.
Ge-Hua-Zhou~\cite{GHZ2} provided a computational method to search for ideal circle patterns by using the combinatorial Ricci flow.
We refer the reader to \cite{GHZ1,Hu,LLX,LZ,XZ} for other developments concerning combinatorial flows. In this paper, we explore the combinatorial Calabi flow introduced by Ge~\cite{G0} to find the desired ideal circle pattern metrics. 

A curvature vector $\overline{K}=(k_{1},k_{2},\cdots,k_{N})\in\mathbb{R}^{N}$ is said to be hyperbolic (resp. Euclidean) \textbf{attainable} if it satisfies \eqref{E-1-1} and \eqref{E-1-2} (resp. \eqref{12}).
In hyperbolic background geometry, let us consider the following flow
\begin{equation}\label{E-1-4}
\frac{dr_{i}}{dt}=\sum_{j=1}^{N}(k_{j}-K_{j})\frac{\partial K_{j}}{\partial r_{i}}\sinh^{2} r_{i} \quad (i=1,2,\cdots,N),
\end{equation}
where $r(0)\in \mathbb{R}_{+}^{N}$ is an initial radius vector.
In Euclidean background geometry, consider the flow
\begin{equation}\label{E-1-5}
\frac{dr_{i}}{dt}=\sum_{j=1}^{N}(k_{j}-K_{j})\frac{\partial K_{j}}{\partial r_{i}} r^{2}_{i} \quad (i=1,2,\cdots,N),
\end{equation}
where $r(0)\in \mathbb{R}_{+}^{N}$ is an initial radius vector. Using Lyapunov functions, we get the following result.

\begin{theorem}\label{T-1-2}
Assume that $\Theta$ is an angle function satisfying $\mathbf{(C1)}$ and $\overline{K}$ is hyperbolic $(resp.\,\,Euclidean)$ attainable. Then the flow \eqref{E-1-4} $(resp.\,\,\eqref{E-1-5})$ exists for all time and converges exponentially fast to a radius vector which produces a hyperbolic $(resp.\,\,Euclidean)$  ideal circle pattern metric on surfaces whose discrete curvatures are assigned by $\overline{K}$.
\end{theorem}

The paper is organized as follows. In Section \ref{S-2}, we introduce several properties of circle pattern.
In Section \ref{S3}, we prove Theorems \ref{0703} and \ref{19}.
We give the proofs of Theorems \ref{T-4-1} and \ref{T-1-2} in Section \ref{S4}.

\section{Preliminaries}\label{S-2}
In this section, we present several results regarding circle pattern.
\begin{lem}\cite[Lemma 1.1]{GHZ2}
Given $\Theta_{k}\in(0,\pi)$ and two positive numbers $r_i$, $r_j$, there exists a configuration of two intersection circles in both Euclidean and hyperbolic geometries, unique up to isometries, having radii $r_i$, $r_j$ and meeting exterior intersection angle $\Theta_{k}$.
\end{lem}

\begin{figure}[htbp]
\centering
\includegraphics[height=4.0cm]{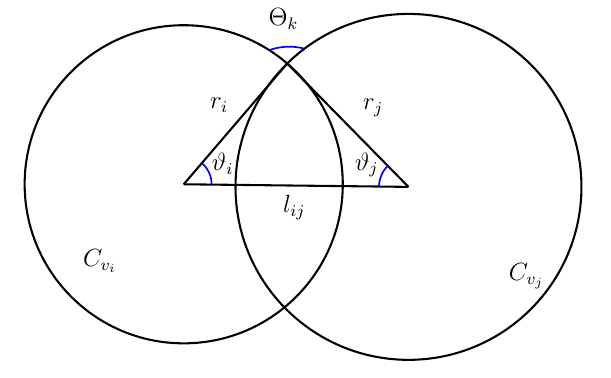}
\caption{A two-circle configuration.}
\label{F1}
\end{figure}

Let $\triangle_{{v}_{i}v_{j}v_{k}}$ be the triangle by connecting the centers $v_i$,$v_j$ and one of the intersection points of two circles (see Figure \ref{F1}).
Let $\vartheta_\eta$ be the inner angle of $\triangle_{{v}_{i}v_{j}v_{k}}$ at the centers $v_\eta$ for $\eta\in\{i,j\}$. 
The following result is due to Ge-Hua-Zhou~\cite{GHZ2}.

\begin{lem}\cite[Lemma 1.2]{GHZ2}\label{L-2-1}
Let $\Theta_k\in(0,\pi)$ be fixed.
\begin{itemize}
\item[(i)] In hyperbolic background geometry,
\[
\frac{\partial\vartheta_{i}}{\partial r_{i}}<0, \quad \sinh r_{j}\frac{\partial\vartheta_{i}}{\partial r_{j}}=\sinh r_{i}\frac{\partial\vartheta_{j}}{\partial r_{i}}>0. 
\]
\item[(ii)] In Euclidean background geometry,
\[
\frac{\partial\vartheta_{i}}{\partial r_{i}}<0, \quad r_{j}\frac{\partial\vartheta_{i}}{\partial r_{j}}=r_{i}\frac{\partial\vartheta_{j}}{\partial r_{i}}>0.
\]
\end{itemize}
\end{lem}

\begin{lem}\cite[Lemma 1.5]{GHZ2}\label{23}
In hyperbolic background geometry, for $\epsilon>0$, there exists a positive number $M$ such that $\vartheta_i\,<\,\epsilon$ whenever $r_i>M$.
\end{lem}



Recall that $\mathcal{T}$ is a triangulation of surface $S$ with vertices $V=\big\{v_1,v_2,\cdots,v_{N}\big\}$.
For a radius vector $r=(r_{1},r_{2},\cdots,r_{N})$, let $u_{i}=\ln\tanh(r_{i}/2)$ (resp. $u_{i}=\ln r_{i}$) in hyperbolic (resp. Euclidean) background geometry for $i=1,2,\cdots,N$. Then the discrete curvatures $K_{1},K_{2},\cdots,K_{N}$ are smooth functions of $u=(u_{1},u_{2},\cdots,u_{N})$. This gives the curvature map $Th$ in terms of $u$:
\[
\begin{aligned}
Th: \qquad {\mathbb{R}}_{-}^{N}\,(\text{resp.}\,\,{\mathbb{R}}^{N}) \quad\,\, &\longrightarrow \,\qquad\qquad\mathbb{R}^{N}\qquad\qquad\\
\big(u_1,u_2,\cdots,u_{N}\big)\quad&\longmapsto \quad\big(K_1, K_2,\cdots,K_{N}\big).\\
\end{aligned}
\]

The following result is obtained by Ge-Hua-Zhou~\cite{GHZ2}.
\begin{lem} \cite[Lemmas 2.1 and 3.1]{GHZ2} \label{L-2-1-2}
Let $Th$ be the curvature map.
\begin{itemize}
\item[(i)] In hyperbolic background geometry, the Jacobian matrix $L$ of $Th$ in terms of $u$ is symmetric and positive definite.
\item[(ii)] In Euclidean background geometry, the Jacobian matrix $L$ of $Th$ in terms of $u$ is symmetric and positive definite when restricted in
      \[
      \Upsilon=\left\{u\in \mathbb{R}^{N}:\sum_{i=1}^{N}u_{i}=d\right\}\quad(\text{$d$ is a given real number}),
      \]
      where the null space of $L$ is $\{k(1,1,\cdots,1):k\in \mathbb{R}\}$.
\end{itemize}
\end{lem}


Recall that $K_i$ is the discrete curvature at $v_i$ for $i=1,2,\cdots,N$. We write $j\thicksim i$ if the vertices $v_i$ and $v_j$ are adjacent. Using geometry arguments, Ge-Hua \cite[Lemmas 2.2 and 2.3]{G2} obtained an initial version of the following estimate. 
Applying Glickenstein-Thomas's result (see \cite[Proposition 9]{GT}), Wu-Xu \cite{WX} and Li-Luo-Xu \cite{LLX}  also independently provided different proofs of the following lemma.

\begin{lem}\label{L-2-3}
In hyperbolic background geometry, for any constant $\rho>1$, there exists a constant $M>0$ such that
\[
\rho\sum_{j\thicksim i}\frac{\partial K_{j}}{\partial u_{i}}+(\rho-1)\frac{\partial K_{i}}{\partial u_{i}}>0
\]
when $r_{i}>M$.
\end{lem}

\section{Combinatorial Calabi flow in hyperbolic background geometry}\label{S3}

\subsection{Hyperbolic background geometry}\noindent
Recall that $\mathcal{T}$ is a triangulation of surface $S$ with vertices $V=\big\{v_1,v_2,\cdots,v_{N}\big\}.$ For $i=1,2,\cdots,N$, let $$r=(r_{1},r_{2},\cdots,r_{N})$$ be a radius vector and  $u_{i}=\ln\tanh\left({r_{i}}/{2}\right).$
Assume that $\overline{K}=(k_{1},k_{2},\cdots,k_{N})$ is a hyperbolic attainable vector.
By changing variables, the combinatorial Calabi flow \eqref{E-1-4} is equivalent to the flow
\begin{equation}\label{E-3-1}
\frac{du_{i}}{dt}=\sum_{j=1}^{N}(k_{j}-K_{j})\frac{\partial K_{j}}{\partial u_{i}}
\quad (i=1,2,\cdots,N),
\end{equation}
where $u(0) \in \mathbb{R}_{-}^{N}$ is an initial vector.

Following Colin de Verdi\`{e}re~\cite{Colin}, Chow-Luo \cite{ChowLuo-jdg} and Ge-Hua-Zhou \cite{GHZ2}, we consider the 1-form
\[
\alpha=\sum_{i=1}^{N}(K_{i}-k_{i})du_{i}.
\]
It is easy to verify that $\alpha$ is a closed 1-form, which implies that
\[
\Psi(u)=\int_{u(0)}^{u} \sum_{i=1}^{N} (K_{i}-k_{i})du_{i}
\]
for $u=(u_{1},u_{2},\cdots,u_{N})\in \mathbb{R}_{-}^{N}$ is well-defined.
Since the Hessian matrix of $\Psi$ is equal to the positive definite Jacobian matrix $L$ mentioned in Lemma \ref{L-2-1-2}, $\Psi$ is strictly convex.
Meanwhile, Theorem \ref{T-1-1} implies that there exists a unique point $r^{\ast}\in \mathbb{R}_{+}^{N}$ such that $K_{i}(r^{\ast})=k_{i}$ for $i=1,2,\cdots,N$. Let $u^{\ast}\in \mathbb{R}_{-}^{N}$ be the vector corresponding to $r^{\ast}$ satisfying $K_{i}(u^{\ast})=k_{i}$.
A direct computation gives
\[
\frac{\partial\Psi}{\partial u_{i}} \,\bigg|_{u=u^{\ast}}=K_{i}(u^{\ast})-k_{i}=0.
\]
Hence $u^{\ast}$ is a critical point of $\Psi$. Motivated by Lyapunov theory \cite{BK,CYR,T3}, we introduce the following function
\[
\Lambda(u)=\Psi(u)-\Psi(u^\ast).
\]

\begin{lem} \label{L-3-1} In hyperbolic background geometry, we have
\[
\lim_{\|u\|\to+\infty}\Lambda(u)=+\infty.
\]
\end{lem}

To prove Lemma~\ref{L-3-1}, we need the following lemma.
\begin{lem}\cite[Lemma 2.8]{GHZ1}\label{L-3-4-1}
Let $f$ be a smooth strictly convex function defined in a convex set $\Omega$ with a critical point $p\in \Omega$. Then the following properties hold:
\begin{itemize}
\item[(i)] $p$ is the unique global minimum point of $f$.
\item[(ii)] If $\Omega$ is unbounded, then $$\lim_{\|x\| \to +\infty} f(x)=+\infty.$$
\end{itemize}
\end{lem}

\begin{proof}[Proof of Lemma~\ref{L-3-1}]
By the above discussion, $\Psi$ is strictly convex in $\mathbb{R}_{-}^{N}$ and $u^{\ast}$ is a critical point. It follows from Lemma~\ref{L-3-4-1} that
\[
\lim_{\|u\| \to +\infty} \Psi(u)=+\infty.
\]
Hence $\Lambda(u)\to+\infty$.
\end{proof}

\subsection{Long time existence}
In what follows, we prove the existence part of Theorem \ref{T-1-2}.
\begin{theorem}\label{0703}
The flow~\eqref{E-3-1} exists for all time.
\end{theorem}
\begin{proof}
Since each $K_{i}$ is a smooth function of $u$, $(K_{1},K_{2},\cdots,K_{N})$ is locally Lipschitz continuous. By ODE theory \cite{T3}, the flow \eqref{E-3-1} has a unique solution $u(t)$ in $[0,\epsilon)$ for some $\epsilon>0$. Hence $u(t)$ exists in a maximal time interval $[0,T_{0})$ with $0<T_{0}\leq +\infty$.
Assume that $T_{0}$ is finite. Then for at least one $i\in \{1,2,\cdots,N\}$, there exists $t_{m}\to T_{0}$ such that
\[
r_{i}(t_{m})\to 0 \  \text{or}\    r_{i}(t_{m})\to +\infty.
\]

For the first case, we have $u_{i}(t_{m})\to -\infty$.
Lemma \ref{L-3-1} implies
\[
\Lambda(u(t_{m}))\to+\infty.
 \]
Since $L$ is positive definite, a calculation yields
\[
\begin{aligned}
\frac{d\Lambda}{dt}
&=\sum_{i=1}^{N} \frac{\partial \Psi}{\partial u_{i}}\frac{du_{i}}{dt}\\
&=-\sum_{i=1}^{N}\sum_{j=1}^{N}(K_{i}-k_{i})\frac{\partial K_{j}}{\partial u_{i}}(K_{j}-k_{j})\\
&=-(K_{1}-k_{1},\cdots,K_{N}-k_{N})L(K_{1}-k_{1},\cdots,K_{N}-k_{N})^{T}\\
&\leq0.
\end{aligned}
\]
It follows that
\[
\Lambda(u(t_{m}))\leq\Lambda(u(0)),
\]
which leads to a contradiction.

We now consider the second case. 
Let $k_{\max}$ and $k_{\min}$ be the maximum and minimum values among $\{k_{1}, k_{2},\cdots, k_{N}\}$, respectively.
By Lemma \ref{23},
there exists $\eta\in(0,1)$ and $M_{1}>0$ for any $r_{i}>M_{1}$ such that
\begin{equation}\label{E-3-5}
K_{i}>2\pi \eta>k_{\max}.
\end{equation}
Meanwhile, Lemma \ref{L-2-1} implies
\begin{equation}\label{E-3-4}
\frac{\partial K_{j}}{\partial u_{i}}<0\  \text{and} \ \frac{\partial K_{i}}{\partial u_{i}}>0.
\end{equation}
A direct computation yields
\[
\begin{aligned}
\frac{du_{i}}{dt}
&=\sum_{j=1}^{N}(k_{j}-K_{j})\frac{\partial K_{j}}{\partial u_{i}}\\
&=\sum_{j\thicksim i}(k_{j}-K_{j})\frac{\partial K_{j}}{\partial u_{i}}+(k_{i}-K_{i})\frac{\partial K_{i}}{\partial u_{i}}\\
&<-\left[(2\pi-k_{\min})\sum_{j\thicksim i}\frac{\partial K_{j}}{\partial u_{i}}+(2\pi\eta-k_{\max})\frac{\partial K_{i}}{\partial u_{i}}\right].
\end{aligned}
\]
Note that
\[
\frac{2\pi-k_{\min}}{2\pi\eta-k_{\max}}-1>0.
\]
It follows from Lemma \ref{L-2-3} that there exists $M_{2}>0$ such that for $r_i>M_2$,
\[
(2\pi-k_{\min})\sum_{j\thicksim i}\frac{\partial K_{j}}{\partial u_{i}}+(2\pi\eta-k_{\max})\frac{\partial K_{i}}{\partial u_{i}}>0.
\]

Let $M=\max\{M_{1},M_{2}\}.$ Then $\frac{du_{i}}{dt}<0,$ which implies $$\frac{dr_{i}}{dt}<0.$$
Since $r_{i}(t_{m})\to +\infty$, we can choose a sufficiently large $n$ such that $r_{i}(t_{n})>M.$
Let
\[
a=\inf_{s}\left\{s<t_{n}: r_{i}(t)>M\quad(\forall t\in[s,t_{n}])\right\}.
\]
Then we can verify that $$r_{i}(a)=M.$$
On the other hand, we have $\frac{dr_{i}(t)}{dt}<0$
for $t\in[a,t_{n}].$ Therefore, we see that $$r_{i}(a)>r_{i}(t_{n})>M,$$
which contradicts to $r_{i}(a)=M.$

By the above discussions, we know that $r_{i}(t)$ is uniformly bounded for $i=1,2,\cdots,N$.
Hence the flow \eqref{E-3-1} exists for all time and $u(t)$ stays in a compact set of $\mathbb{R}_{-}^{N}$.
\end{proof}

\subsection{Convergence}
Using an energy function, we prove the convergence of the flow.
\begin{theorem}\label{19}
Suppose that $\overline{K}=(k_{1},k_{2},\cdots,k_{N})$ is hyperbolic attainable.
The flow~\eqref{E-3-1} converges exponentially fast to a radius vector which produces a hyperbolic ideal circle pattern metric on surfaces with discrete curvatures $k_{1},k_{2},\cdots, k_{N}$.
\end{theorem}
\begin{proof}
Define 
\[
C(u)=\sum_{i=1}^{N}(K_{i}-k_{i})^{2}
\]
as an energy function. Since $u(t)$ stays in a compact set of $\mathbb{R}_{-}^{N}$ and $L$ depends on $u$ continuously, there exists $\lambda_{0}>0$ such that
\[
\begin{aligned}
C^\prime(u)&=2\sum_{i=1}^{N}\sum_{j=1}^{N}(K_{j}-k_{j})\frac{\partial K_{j}}{\partial u_{i}}\frac{du_{i}}{dt}\\
&=-2(K_{1}-k_{1},\cdots,K_{N}-k_{N})L^{2}(K_{1}-k_{1},\cdots,K_{N}-k_{N})^{T}\\
&\leq-2\lambda_{0}\sum_{i=1}^{N}(K_{i}-k_{i})^2\\
&=-2\lambda_{0} C(u).
\end{aligned}
\]
It follows that
\[
\sum_{i=1}^{N}(K_{i}-k_{i})^2 =C(u)\leq C(u(0))e^{-2\lambda_{0} t}.
\]
Therefore, we get
\[
|K_{i}-k_{i}|\leq \sqrt{C(u(0))}e^{-\lambda_{0}t}
\]
for $i=1,2,\cdots,N$.
Note that $\left|\frac{\partial K_{j}}{\partial u_{i}}\right|$ is uniformly bounded.
Then we have
\[
|u_{i}-u^{\ast}_{i}|=\left|\sum_{j=1}^{N}\int_{\infty}^{t}(K_{j}-k_{j})\frac{\partial K_{j}}{\partial u_{i}}dt\right|\leq \lambda e^{-\lambda_{0}t}
\]
for some positive number $\lambda$.
As $t\to+\infty$, we derive that $u$ converges exponentially fast to $u^{\ast}$.
We thus complete the proof of Theorem \ref{19}.
\end{proof}

\section{Combinatorial Calabi flow in Euclidean background geometry}\label{S4}\noindent

\subsection{Euclidean background geometry}
Let $u_{i}=\ln r_{i}$ for $i=1,2,\cdots,N$ and $u=(u_{1},u_{2},\cdots,u_{N})$.
Suppose that $\overline{K}=(k_{1},k_{2},\cdots,k_{N})$ is a Euclidean attainable vector.
Then the flow \eqref{E-1-5} transforms into the equivalent flow
\begin{equation}\label{E-3-6}
\frac{du_{i}}{dt}=\sum_{j=1}^{N}(k_{j}-K_{j})\frac{\partial K_{j}}{\partial u_{i}}
\quad(i=1,2,\cdots,N),
\end{equation}
where $u(0)$ is an initial vector in $\mathbb{R}^{N}$.

A direct calculation gives
\[
\sum_{i=1}^{N}\frac{du_{i}}{dt}=-(1,1,\cdots,1)L(K_{1}-k_{1},\cdots,K_{N}-k_{N})^{T},
\]
where $L$ is the Jacobian matrix mentioned in Lemma \ref{L-2-1-2}. Moreover, the fact that the matrix $L$ has a null space $\{k(1,1,\cdots,1):k\in \mathbb{R} \}$ implies that $$\sum_{i=1}^{N}\frac{du_{i}}{dt}=0.$$
Hence $\sum_{i=1}^{N}u_{i}$ remains a constant. It means that
$u\in \Upsilon$, where
\[
\Upsilon=\left\{u\in \mathbb{R}^{N}:\sum_{i=1}^{N}u_{i}=\sum_{i=1}^{N}u_{i}(0)\right\}.
\]

Let
\[
\Psi(u)=\int_{u(0)}^{u} \sum_{i=1}^{N} (K_{i}-k_{i})du_{i}
\]
be a function defined in the set $\Upsilon$.
A direct calculation shows that the Hessian matrix of $\Psi$ is equal to the Jacobian matrix $L$.
Because $u$ is restricted in $\Upsilon$, it follows from Lemma \ref{L-2-1-2} that $L$ is positive definite.
Similar to the discussion in hyperbolic background geometry, $\Psi$ is well defined and strictly convex.

By Theorem \ref{T-1-1}, we know that there exists a radius vector $r^{\ast}\in \mathbb{R}_{+}^{N}$ such that $K_{i}(r^{\ast})=k_{i}$ for $i=1,2,\cdots,N$.
Note that $$Th(\mathbb{R_{+}^{N}})=Th(\Upsilon).$$
Since $L$ is positive definite, there exists a unique vector $u^{\ast}\in\Upsilon$ corresponding to $r^{\ast}$ such that $$K_{i}(u^{\ast})=k_{i}$$ for $i=1,2,\cdots,N$.
Moreover,
\[
\frac{\partial\Psi}{\partial u_{i}} \,\bigg|_{u=u^{\ast}}=K_{i}(u^{\ast})-k_{i}=0.
\]
Hence $u^{\ast}$ is a unique critical point of $\Psi$.

We now consider the function
\[
\Lambda(u)=\Psi(u)-\Psi(u^\ast).
\]
By a similar argument to the proof of Lemma \ref{L-3-1}, we obtain the following result.
\begin{lem} \label{L-3-3}
In Euclidean background geometry, we have
\[
\lim_{\|u\|\to+\infty}\Lambda(u)=+\infty.
\]
\end{lem}

\subsection{Long time existence and convergence}
Now we prove the following result.
\begin{theorem}\label{T-4-1}
Suppose that $\overline{K}=(k_{1},k_{2},\cdots,k_{N})$ is Euclidean attainable.
The flow \eqref{E-3-6} exists for all time and converges exponentially fast to a radius vector that produces a Euclidean ideal circle pattern metric on surfaces with discrete curvatures $k_{1},k_{2},\cdots,k_{N}$.
\end{theorem}
\begin{proof}
By a similar argument to the hyperbolic background geometry, we prove that the flow \eqref{E-1-5} has a unique solution $u(t)$ in a maximal time interval $[0,T_{0})$ with $0<T_{0}\leq +\infty$. We claim that $T_{0}=+\infty$. Otherwise, assume that $T_{0}$ is finite. Then for at least one $i\in\{1,2,\cdots,N\}$, there exists a sequence $\{t_{m}\}$ with limit $T_{0}$ such that
\[
r_{i}(t_{m})\to 0 \  \text{or}\  r_{i}(t_{m})\to +\infty.
\]
It shows that \[u_{i}(t_{m})\to -\infty \  \text{or}\  u_{i}(t_{m})\to +\infty.\]
No matter which case occurs, it follows from Lemma \ref{L-3-3} that $\Psi(u(t_{m}))\to+\infty,$
which implies $$\Lambda(u(t_{m}))\to+\infty.$$
Nevertheless, a direct computation gives
\[
\begin{aligned}
\frac{d\Lambda}{dt}
&=\sum_{i=1}^{N} \frac{\partial \Psi}{\partial u_{i}}\frac{du_{i}}{dt}\\
&=-\sum_{i=1}^{N}\sum_{j=1}^{N}(K_{j}-k_{j})\frac{\partial K_{j}}{\partial u_{i}}(K_{i}-k_{i})\\
&=-(K_{1}-k_{1},\cdots,K_{N}-k_{N})L(K_{1}-k_{1},\cdots,K_{N}-k_{N})^{T}\\
&\leq0.
\end{aligned}
\]
As a result, we have
\[
\Lambda(u(t_{m}))\leq\Lambda(u(0)),
\]
which leads to a contradiction.
Hence the flow \eqref{E-3-6} never touches the boundary of $\Upsilon\subset\mathbb{R}^{N}$ in any finite time interval.
Namely, the flow exists for all time and $u(t)$ stays in a compact set of $\Upsilon$.

The proof of exponentially convergence of the flow \eqref{E-3-6} is similar to the argument of the flow \eqref{E-3-1}.
\end{proof}

Finally, we give the proof of Theorem \ref{T-1-2}.
\begin{proof}[Proof of Theorem \ref{T-1-2}]
Since the flow \eqref{E-1-4} (resp. \eqref{E-1-5}) is equivalent to the flow \eqref{E-3-1}  (resp. \eqref{E-3-6}), the conclusion of Theorem \ref{T-1-2} is an immediate consequence of Theorems \ref{0703}, \ref{19} and \ref{T-4-1}.
\end{proof}

\vskip .05in
	\noindent{\bf  Acknowledgements.}
    The present investigation was supported by the \textit{Natural Science Foundation of Hunan Province} under Grant no. 2022JJ30185 of the P. R. China.
The authors thank Prof. Ze Zhou for his stimulating discussions and Prof. Yueping Jiang for his invaluable help.

\vskip .05in
\noindent{\bf Conflicts of interest.} The authors declare that they have no conflict of interest.

\vskip .05in
\noindent{\bf Data availability statement.}  Data sharing is not applicable to this article as no datasets were generated or analysed during the current study.

\vskip .05in
\noindent{\bf Statement of independent research outcomes.}
It has come to the authors' attention that Prof. Xiaoxiao Zhang at Beijing Wuzi University has independently arrived at similar conclusions, coinciding with the authors' work.

\end{document}